\documentclass[12pt]{article}

\usepackage{amsmath,amsthm,amsfonts,amssymb,amscd,cite}
\usepackage{tikz}
\usepackage{color}

\newtheorem{theorem}{Theorem}[section]

\newtheorem{lemma}[theorem]{Lemma}

\begin{document}

\title{Anti‑Ramsey Numbers for Spanning Linear Forests of 3‑Vertex Paths and Matchings}

\author{
Ali Ghalavand$^{a,}$\thanks{Corresponding author, email:\texttt{alighalavand@himis-sz.cn}}
\and  Xueliang Li$^{b,c}$}

\maketitle

\begin{center}
$^a$ Hetao Institute of Mathematics and Interdisciplinary Sciences (HIMIS), Shenzhen 518000, Guangdong, P. R. China\\
$^b$ Center for Combinatorics and LPMC, Nankai University, Tianjin 300071, China\\
$^c$ School of Mathematical Sciences, Xinjiang Normal University, Urumchi, Xinjiang 830017, China
\end{center}

\begin{abstract}
A subgraph in an edge-colored graph is called rainbow if all its edges have distinct colors. For a graph \(G\) and an integer \(n\), the anti‑Ramsey number \(\operatorname{AR}(n,G)\) is the maximum number of colors in an edge‑coloring of \(K_n\) that contains no rainbow copy of \(G\). We study \(\operatorname{AR}(n, kP_3 \cup tP_2)\), where \(kP_3 \cup tP_2\) is the linear forest of \(k\) disjoint paths on three vertices and a matching of size \(t\). Recently, Jie and Jin [Discrete Appl. Math. 386 (2026) 30–57] determined this number for \(k\ge 2\), \(t\ge\frac{k^2-3k+4}{2}\) and \(n=2t+3k\). Here we solve the spanning case \(n=3k+2t\) for all \(k\ge1\), \(t\ge2\) with no extra restrictions.
\end{abstract}

%\medskip\noindent
%\textbf{AMS Math.\ Subj.\ Class.\ (2020)}:

\section{Introduction}
In this study, we examine the anti-Ramsey number of forests composed solely of paths of lengths two and three. We will begin by defining key concepts and discussing their historical context.
Let \( G \) be a finite and simple graph with a vertex set \( V(G) \) and an edge set \( E(G) \). An edge-coloring of a graph \( G \) is a mapping \( c \) that assigns a color to each edge in \( E(G) \). The graph \( G \) is considered to be rainbow under the edge coloring \( c \) if all of its edges have different colors.
For an edge \( e \) in \( G \), the notation \( c(e) \) indicates the color assigned to the edge \( e \) by the map \( c \). For a subset \( E' \) of \( E(G) \), the notation \( c(E') \) represents the set of colors assigned to the edges in \( E' \), that is, \( c(E') = \{ c(e) : e \in E' \} \). Similarly, for a subgraph \( H \) of \( G \), the notation \( c(H) \) represents the set of colors assigned to the edges in \( H \), specifically, \( c(H) = \{ c(e) : e \in E(H) \} \).   
%%%%
Let \( V' \) be a subset of the vertex set \( V(G) \) of a graph \( G \). The subgraph of \( G \) that is induced by \( V' \) is denoted by \( G[V'] \). This subgraph has the vertex set \( V' \), and two vertices in it are considered adjacent if and only if they are adjacent in the original graph \( G \).
For another graph \( G' \), the disjoint union of \( G \) and \( G' \), denoted as \( G \cup G' \), is a new graph that includes both \( G \) and \( G' \). The vertex set of this new graph is given by \( V(G \cup G') = V(G) \cup V(G') \), while the edge set is defined as \( E(G \cup G') = E(G) \cup E(G') \). Additionally, for any positive integer \( a \), the notation \( aG \) refers to the disjoint union of \( a \) copies of the graph \( G \).
Let \( t \) be a positive integer. We will use the notation \( [t] \) to represent the set of natural numbers \( \{ 1, 2, \ldots, t \} \). For clarity, we define \( [0] = \emptyset \). The notations \( K_t \), \(C_t\), and \( P_t \) refer to the complete graph, the cycle graph, and the path graph on \( t \) vertices, respectively.  
%%%

Let \( n \) be a positive integer. The anti-Ramsey number of a graph \( G \), denoted as \( AR(n, G) \), represents the maximum number of colors in an edge-coloring of the complete graph \( K_n \) that does not contain a rainbow copy of \( G \). This concept was first introduced by Erdős et al. in 1975 \cite{Er1}, who established its connection to Turán numbers and conjectured the anti-Ramsey numbers for paths and cycles. Inspired by their work, a particularly active area of research has focused on characterizing anti-Ramsey numbers for graphs made up of multiple small, disjoint components. A key example of this is a matching, which consists of disjoint edges; the anti-Ramsey numbers for matchings in complete graphs are now fully understood, as referenced in \cite{p1-2,p1-5,p1-7,p1-9,p1-14}. In recent years, the focus has shifted from single matchings to more complex structures. Notably, researchers have begun exploring anti-Ramsey numbers for graphs formed by the vertex-disjoint union of small connected graphs and matchings.
Simonovits and Sós \cite{p2-10} proved that if \( t \geq 2 \) and \( n \) is sufficiently large, then the following holds:
\[
AR(n,P_t) = {\left\lfloor\frac{t-1}{2}\right\rfloor\choose2} + \left( \left\lfloor\frac{t-1}{2}\right\rfloor - 1 \right) \left(n - \left\lfloor\frac{t-1}{2}\right\rfloor + 1\right) + 1 + \epsilon,
\]
where \( \epsilon = 1 \) if \( t \) is even and \( \epsilon = 0 \) otherwise.
Montellano-Ballesteros and Neumann-Lara \cite{p3-6} proved that for any \( n \geq t \geq 3 \),
\[
AR(n,C_t) = {t-1\choose2} \left\lfloor\frac{n}{t-1}\right\rfloor + \left\lceil\frac{n}{t-1}\right\rceil + {n \mod (t-1) \choose 2}.
\]
Schiermeyer \cite{p2-9} showed that for all \( t \geq 2 \) and \( n \geq 3t + 3 \), the following holds:
\[
AR(n,tP_2) = {t-2 \choose 2} + (t-2)(n - t + 2) + 1.
\]
Then, Chen et al. \cite{p1-2} and Fujita et al. \cite{p1-5} demonstrated that for all \( t \geq 2 \) and \( n \geq 2t + 1 \):
\[
AR(n,tP_2) = \begin{cases} 
(t-2)(2t-3) + 1 & \text{when } n \leq \frac{5t-7}{2}, \\ 
(t-2)(n - \frac{t-1}{2}) + 1 & \text{when } n \geq \frac{5t-7}{2}. 
\end{cases}
\]
Finally, the remaining case \( n = 2t \) was addressed by Haas and Young \cite{p1-7}, who confirmed the conjecture made in \cite{p1-5}. They showed that if \( n = 2t \), then:
\[
AR(n,tP_2) = \begin{cases} 
\frac{1}{2}(t-2)(3t+1) + 1 & \text{when } 3 \leq t \leq 6, \\ 
(t-2)(2t-3) + 2 & \text{when } t \geq 7. 
\end{cases}
\]
Bialostocki et al. \cite{p1-1} examined the anti-Ramsey number of graphs with at most four edges and proved that \( AR(n,P_3 \cup P_2) = 2 \) when \( n \geq 5 \) and \( AR(n,P_3 \cup 2P_2) = n \) when \( n \geq 7 \).
For more general graph classes, Gilboa and Roditty \cite{p1-6} provided upper bounds on the anti-Ramsey numbers of \( L \cup tP_2 \) and \( L \cup tP_3 \) under specific conditions. They proved that, for sufficiently large \( n \),
\begin{enumerate}
 \item[$\bullet$] $AR(n,P_{k+1}\cup tP_3)=(t+\lfloor\frac{k}{2}\rfloor-1)(n-\frac{t+\lfloor\frac{k}{2}\rfloor}{2})+1+(k \mod 2)$ when $k\geq3$;
 \item[$\bullet$] $AR(n,kP_3\cup tP_2)=(k+t-2)(n-\frac{t+k-1}{2})+1$ when $k,t\geq2$;
 \item[$\bullet$] $AR(n,P_2\cup tP_3)=(t-1)(n-\frac{t}{2})+1$ when $t\geq1$;
  \item[$\bullet$] $AR(n,P_3\cup tP_2)=(t-1)(n-\frac{t}{2})+1$ when $t\geq2$;
  \item[$\bullet$] $AR(n,P_4\cup tP_2)=t(n-\frac{t+1}{2})+1$ when $t\geq1$;
   \item[$\bullet$] $AR(n,C_3\cup tP_2)=t(n-\frac{t+1}{2})+1$ when $t\geq1$;
 \item[$\bullet$] $AR(n, tP_3)=(t-1)(n-\frac{t}{2})+1$ when $t\geq1$.
  \end{enumerate}
%%%%
He and Jin \cite{p1-8} established that if \( t \geq 2 \) and \( n \geq 2t + 3 \), then 
\[
AR(n, P_3 \cup tP_2) = 
\begin{cases} 
t(2t-1) + 1 & \text{when } 2t + 3 \leq n \leq \left\lfloor \frac{5t + 2}{2} + \frac{1}{t-1} \right\rfloor, \\
\frac{1}{2} (t-1)(2n - t) + 1 & \text{when } n \geq \left\lceil \frac{5t + 2}{2} + \frac{1}{t-1} \right\rceil. 
\end{cases}
\]
They also proved that if \( t \geq 2 \) and \( n \geq 2t + 7 \), then 
\[
AR(n, 2P_3 \cup tP_2) = 
\begin{cases} 
(t+1)(2t + 3) + 1 & \text{when } 2t + 7 \leq n \leq \left\lfloor \frac{5t + 11}{2} + \frac{3}{t} \right\rfloor, \\
\frac{1}{2} t(2n - t - 1) + 1 & \text{when } n \geq \left\lceil \frac{5t + 11}{2} + \frac{3}{t} \right\rceil. 
\end{cases}
\]
Fang et al. \cite{p1-4} considered \( F \) to be a linear forest with components of order \( p_1, p_2, \ldots, p_t \), where \( t \geq 2 \) and \( p_i \geq 2 \) for \( 1 \leq i \leq t \), and at least one \( p_i \) is even. They proved that for sufficiently large \( n \), 
\[
AR(n, F) = \left( \sum_{i \in [t]} n \left\lfloor \frac{p_i}{2} \right\rfloor - \epsilon \right) n + {O}(1),
\] 
where \( \epsilon = 1 \) if all \( p_i \) are odd, and \( \epsilon = 2 \) otherwise.
Following this, Xie et al. \cite{p1-14} established an exact expression for the anti-Ramsey numbers of linear forests containing even components. They proved that if at least one \( p_i \) is even, where \( t \geq 2 \) and \( p_i \geq 2 \) for all \( 1 \leq i \leq t \), then for sufficiently large \( n \), 
\[
AR(n, F) = \binom{\sum_{i \in [t]} \left\lfloor \frac{p_i}{2} \right\rfloor - 2}{2} + \left( \sum_{i \in [t]} \left\lfloor \frac{p_i}{2} \right\rfloor - 2 \right) \left( n - \sum_{i \in [t]} \left\lfloor \frac{p_i}{2} \right\rfloor + 2 \right) + 1 + \epsilon,
\]
where \( \epsilon = 1 \) if exactly one \( p_i \) is even, or \( \epsilon = 0 \) if at least two \( p_i \) are even. Note that this expression does not account for the case when \( n = \sum_{i \in [t]} p_i \).
Define
\[
\lambda = \frac{9k+5t-7}{2} + \frac{k(k+1)}{2(k+t-2)}.
\]
Jie et al.~\cite{p1-10} proved that for \(k\ge 2\), \(t\ge \frac{k^2-k+4}{2}\), and \(n\ge 3k+2t+1\),
\[
\operatorname{AR}(n,kP_3\cup tP_2)=
\begin{cases}
	\frac12(3k+2t-3)(3k+2t-4)+1, & n\le\lfloor\lambda\rfloor,\\[4pt]
	\frac12(k+t-2)(2n-k-t+1)+1, & n\ge\lceil\lambda\rceil.
\end{cases}
\]
Two of the authors later in \cite{new-1} treated  the spanning case \(n=2t+3k\) with \(k\ge 2\), \(t\ge \frac{k^2-3k+4}{2}\), and proved
\[
\operatorname{AR}(2t+3k,kP_3\cup tP_2)=\frac12(2t+3k-3)(2t+3k-4)+1.
\]
The present paper addresses the complementary spanning case: for all \(k\ge 1\), \(t\ge 2\), and \(n=3k+2t\), we show that no additional relation between \(k\) and \(t\) is needed. Specifically, we prove the following theorem:
\begin{theorem}\label{th3}
For three positive integers \( k \), \( t \), and \( n \), if \( k \geq 1 \), \( t \geq 2 \), and \( n = 3k + 2t \), then
\[
AR(n, kP_3 \cup tP_2) = \frac{1}{2} (3k + 2t - 3)(3k + 2t - 4) + 1.
\]
\end{theorem}

\section{Key Lemma and Theorem}
In this section we collect the two essential ingredients for the proof of the main theorem. The first is a result of He and Jin \cite{p1-8}; the second is a new lemma that we establish here.

\begin{theorem}{\rm \cite{p1-8}}\label{rth1}
For any positive integer $t$, if $t\geq2$ and $\mu=\frac{5t+2}{2}+\frac{1}{t-1}$, then
\[AR(n, P_3 \cup tP_2) =\left\{\begin{array}{ll}
  t(2t-1)+1&{\rm ~when~} 2t+3 \leq n\leq\lfloor\mu\rfloor,  \\
 (t-1)(n-\frac{t}{2})+1&  {\rm ~when~ } n\geq\lceil\mu\rceil.
 \end{array}\right.\]
\end{theorem}

\begin{lemma}\label{lm1}
For three positive integers \( k \), \( t \), and \( n \), where $k\geq1$, \(t \geq 2\), and \(n = 3k + 2t\), let \( c \) be an edge-coloring of \( K_n \) such that 
\[
|c(K_n)| \geq \frac{1}{2} (3k + 2t - 3)(3k + 2t - 4) + 2.
\]
If \( K_{n-3} \) is a subgraph of \( K_n \) such that 
$|c(K_{n-3})|$ is maximized, then it holds that
\[
|c(K_{n-3})| \geq \frac{1}{2}(3k + 2t - 6)(3k + 2t - 7) + 2.
\]
\end{lemma}
\begin{proof}
  For three positive integers \(k\), \(t\), and \(n\) where $k\geq1$, \(t \geq 2\), and \(n = 3k + 2t\), let \(c\) be an edge-coloring of \(K_n\) such that 
\[
|c(K_n)| \geq \frac{1}{2}(3k + 2t - 3)(3k + 2t - 4) + 2.
\]
Assume that \(G\) is a rainbow spanning subgraph of size \(|c(K_n)|\) of \(K_n\), and let \(K_{n-3}\) be a subgraph of \(K_n\) such that \(|E(K_{n-3}) \cap E(G)|\) is maximized. To prove our lemma, it suffices to demonstrate that 
\[
|E(K_{n-3}) \cap E(G)| \geq \frac{1}{2}(3k + 2t - 6)(3k + 2t - 7) + 2.
\]
To establish this,  we will analyze the situation in two distinct cases as follows:

\noindent
{\bf Case I:} For every vertex \(v\) of \(K_{n-3}\), it holds that
   \[
   |\{uv : u \in V(K_{n-3})\} \cap E(G)| \geq 3k + 2t - 6.
   \]
   In this case, since $k\geq1$ and \(t \geq 2\), it follows that 
   \[
   |E(K_{n-3}) \cap E(G)| \geq \frac{1}{2}(3k + 2t - 3)(3k + 2t - 6) \geq \frac{1}{2}(3k + 2t - 6)(3k + 2t - 7) + 2,
   \]
   which is as desired.
   
   \noindent
 {\bf Case II:} For a vertex \(v\) of \(K_{n-3}\), it holds that
   \[
   |\{uv : u \in V(K_{n-3})\} \cap E(G)| \leq 3k + 2t - 7,
   \]
   and it also holds that 
   \[
   |E(K_{n-3}) \cap E(G)| \leq \frac{1}{2}(3k + 2t - 6)(3k + 2t - 7) + 1.
   \]
   In this case, since \(|E(G)| = |c(K_n)| \geq \frac{1}{2}(3k + 2t - 3)(3k + 2t - 4) + 2\), it follows that 
   \[
   |E(G) - E(K_{n-3})| \geq 9k + 6t - 14.
   \]
   Thus, there exists a vertex \(x\) in \(V(K_n) - V(K_{n-3})\) such that 
   \[
   |\{ux : u \in V(K_{n-3})\} \cap E(G)| \geq 3k + 2t - 5.
   \]
   However, this leads to a contradiction regarding the maximality, because using the vertices \(v\) and \(x\), we can construct a graph \(K_{n-3}\) such that 
   \(|E(K_{n-3}) \cap E(G)|\) is larger.
   
Now, based on these two cases, we can conclude that 
\[
|E(K_{n-3}) \cap E(G)| \geq \frac{1}{2}(3k + 2t - 6)(3k + 2t - 7) + 2,
\]
therefore completing the proof.
\end{proof}

\section{Proof of Theorem~\ref{th3}}
For three positive integers \( k \), \( t \), and \( n \), assume $k\geq1$, \( t \geq 2 \) and \( n = 3k + 2t \). To prove our result, we need to confirm two cases: \\
1. \( AR(n, kP_3 \cup tP_2) \geq \frac{1}{2}(3k + 2t - 3)(3k + 2t - 4) + 1 \),\\
2. \( AR(n, kP_3 \cup tP_2) \leq \frac{1}{2}(3k + 2t - 3)(3k + 2t - 4) + 1 \).\\
We will address these cases as follows:

\noindent
{\bf Case 1:} To establish this case, suppose \( c \) is an edge-coloring of \( K_n \) such that for a subgraph \( K_{n-3} \) of \( K_n \), \( K_{n-3} \) is rainbow, and all other edges of \( K_n \) that are not in \( K_{n-3} \) are colored with a new color. It is evident that \( |c(K_n)| = \frac{1}{2}(3k + 2t - 3)(3k + 2t - 4) + 1 \) and that there is no rainbow subgraph of \( K_n \) isomorphic to \( kP_3 \cup tP_2 \). Thus, we conclude that \( AR(n, kP_3 \cup tP_2) \geq \frac{1}{2}(3k + 2t - 3)(3k + 2t - 4) + 1 \).

\noindent
{\bf Case 2:} To prove this case, we need to show that if \( c \) is an arbitrary edge-coloring of \( K_n \) such that \( |c(K_n)| = \frac{1}{2}(3k + 2t - 3)(3k + 2t - 4) + 2 \), then \( K_n \) has a rainbow subgraph isomorphic to \( kP_3 \cup tP_2 \). We will use induction on \( k \). If \( k = 1 \), Theorem \ref{rth1} provides the result. Now, let \( k \geq 2 \), and consider \( K_{n-3} \), a subgraph of \( K_n \), that has a maximized \( |c(K_{n-3})| \). Assume \( S = V(K_n) - V(K_{n-3}) = \{s_1, s_2, s_3\} \). By applying Lemma \ref{lm1}, we find that \( |c(K_{n-3})| \geq \frac{1}{2}(3(k-1) + 2t - 3)(3(k-1) + 2t - 4) + 2 \). Thus, by the inductive hypothesis, the graph \( K_{n-3} \) contains a rainbow subgraph \( H \) such that \( H \cong (k-1)P_3 \cup tP_2 \).
 Now, assume $H_1$ and $H_2$ are two distinct subgraphs of $H$ that $H_1=(k-1)P_3$ and $H_2=tP_2$. In addition, suppose $H_1=\{P_3^i:i\in[k-1]\}$, $H_2=\{P_2^j:j\in[t]\}$, $E(P_3^i)=\{x_1^ix_2^i,x_2^ix_3^i\}$ for $i\in[k-1]$, and   $E(P_2^j)=\{y_1^jy_2^j\}$ for $j\in[t]$. Also, suppose \( G \) is a rainbow spanning subgraph of size \( |c(K_n)| \) from \( K_n \) that satisfies \( E(H) \subseteq E(G) \). To complete the proof in this case, we will examine sixteen scenarios as outlined below. Please note that the notation \( \mathcal{S}_{2.x}(l_1, \ldots, l_h) \), which we may use in Scenario 2.x, represents the largest subset of the edge set of \( K_n \) that we can derive from the conditions \( l_1 \) to \( l_h \), such that \( \mathcal{S}_{2.x}(l_1, \ldots, l_h) \cap E(G) = \emptyset \).\\ 
{\bf Scenario 2.1:} \(| c(K_n[S])- c(H)|\geq2  \). In this scenario, the graph \( K_n[S] \) contains a rainbow subgraph \( F_1 \) that \( F_1 \cong P_3 \) and $c(F_1)\cap c(H)=\emptyset$. Therefore, \( H \cup F_1 \) forms a rainbow subgraph of \( K_n \) that is isomorphic to \( kP_3 \cup tP_2 \). \\
{\bf Scenario 2.2:}  \( |c(K_n[S])| = 3 \) and \( |c(K_n[S]) - c(H)| = 1 \). For \( l \in [k-1] \), it holds that \( |c(K_n[S]) \cap c(P_3^l)| = 2 \). We also assume that \( c(s_1s_2) \notin c(H) \). If one of the following conditions is satisfied, then it is possible to construct a rainbow subgraph of \( K_n \) isomorphic to \( kP_3 \cup tP_2 \).
\begin{enumerate}
  \item There are $i\in[t]$, $i_1\in[2]$, and $z\in\{s_3,x^l_1,x^l_3\}$ such that $zy^i_{i_1}\in E(G)$ and $c(zy^i_{i_1})\neq c(s_1s_2)$.
 \item There are $i \in ([k-1]-\{l\})$, $j \in [t]$, $j_1\in[2]$, $j_2\in\{1,3\}$, $j_3\in(\{1,3\}-\{j_2\})$ and $z\in\{s_3,x^l_1,x^l_3\}$ such that $\{z x_{j_2}^i,x_{j_3}^i y_{j_1}^j\}\subseteq E(G)$ and $c(s_1s_2)\not\in c(\{z x_{j_2}^i,x_{j_3}^i y_{j_1}^j\})$.
  \item There are $i \in ([k-1]-\{l\})$, $j \in [t]$, $j_1\in[2]$, $j_2\in\{1,3\}$ and $z\in\{s_3,x^l_1,x^l_3\}$ such that $\{z x_{2}^i,x_{j_2}^i y_{j_1}^j\}\subseteq E(G)$ and $c(s_1s_2)\not\in c(\{z x_{2}^i,x_{j_2}^i y_{j_1}^j\})$.
  \item There are $i \in ([k-1]-\{l\})$, $j \in [t]$, $j_1\in[2]$, and $z\in\{s_3,x^l_1,x^l_3\}$ such that $\{z x_{1}^i,zx_3^i,x_{2}^i y_{j_1}^j\}\subseteq E(G)$ and $c(s_1s_2)\not\in c(\{z x_{1}^i,zx_3^i,x_{2}^i y_{j_1}^j\})$.
  \item There are $i\in[t]$, $i_1\in[2]$, $i_2\in([2]-\{i_1\})$, and $i_{3}\in\{1,3\}$ such that either $\{x^l_{i_3}s_1,s_1y^i_{i_1},x^l_2y^i_{i_2}\}$, $\{x^l_{i_3}s_1,s_2y^i_{i_1},x^l_2y^i_{i_2}\}$, or $\{x^l_{i_3}s_1,s_1y^i_{i_1},s_2y^i_{i_2}\}$ is a subset of $E(G)$.
  \item There are $i\in[t]$ and $i_1\in[2]$ such that $\{x^l_1s_2,s_1y^i_{i_1}\}$ is a subset of $E(G)$.
   \item There are $i\in[t]$, $i_1\in[2]$, and $i_2\in([2]-\{i_1\})$ such that either $\{x^l_{3}s_2,s_1y^i_{i_1},$ $x^l_2y^i_{i_2}$ $\}$, $\{x^l_{3}s_2,s_2y^i_{i_1},$ $x^l_2y^i_{i_2}$ $\}$, or $\{x^l_{3}s_2,s_1y^i_{i_1},s_2y^i_{i_2}\}$ is a subset of $E(G)$.
   \item There are $i\in[t]$ and $i_1\in[2]$ such that $\{x^l_{1}s_3,s_1y^i_{i_1}\}$ is a subset of $E(G)$.
   \item There are $i\in[t]$ and $i_1\in[2]$ such that $\{x^l_{3}s_3,s_2y^i_{i_1}\}$ is a subset of $E(G)$.
   \item Both \( x_1^lx_3^l\in E(G) \) and $c(x_1^lx_3^l)\neq c(s_1s_2)$ hold.
   \item There are $i\in[t]$ and $i_1\in[2]$ such that \(c(x_1^l x_3^l) \in c(K_n[S])\), $x^l_2y^i_{i_1}\in E(G)$, and $c(x^l_2y^i_{i_1})\neq c(s_1s_2)$.
    \item There are \( i\in[t] \), \( j \in ([t] - \{i\}) \), $j_1\in[2]$, and $j_2\in([2]-\{j_1\})$ such that \( c(x_1^l x_3^l) = c(y_1^i y_2^i) \), $y_{1}^i y_{j_1}^j,y_{2}^i y_{j_2}^j\in E(G)$,  and $c(s_1s_2)\not\in c(\{y_{1}^i y_{j_1}^j,y_{2}^i y_{j_2}^j\})$.
    \item There is $i\in[t]$ such that $c(x_1^lx_3^l)=c(y^i_1y^i_2)$, $y^i_1x_2^l,y^i_2x_2^l\in E(G)$, and $c(s_1s_2)\not\in c(\{y^i_1x_2^l,y^i_2x_2^l\})$.
    \item There are $i\in[t]$, \( j \in ([k - 1] - \{ l \}) \), $i_1\in[2]$, \( j_1\in( [3]-\{2\}) \) such that \( c(x_1^l x_3^l) = c(x^j_{j_1} x^j_2) \), $x^j_{j_1} y^i_{i_1}\in E(G)$, and $c(s_1s_2)\neq c(x^j_{j_1} y^i_{i_1})$.
     \item There are \( i \in( [k-1] - \{l\}) \) and \( i_1 \in ([3] - \{2\}) \) such that \( c(x_1^l x_3^l) = c(x^i_{i_1} x^i_2) \),  \( x_1^i x_3^i\in E(G) \), and $c(s_1s_2)\neq c(x_1^i x_3^i)$.
\end{enumerate}
If none of the above conditions hold, we can observe that $|\mathcal{S}_{2.2}(1)| \geq 6t$, $|\mathcal{S}_{2.2}(2,3)| \geq 9(k-2)$, and $|\mathcal{S}_{2.2}(4,\ldots,15) \cup \{s_1s_3,s_2s_3\}| \geq 11$. Thus, we find that
\begin{align*}
|c(K_n)|&\leq  \frac{1}{2}(3k+2t)(3k+2t-1)-\big(6t+9(k-2)+11\big)\\
&=\frac{1}{2}(3k+2t-3)(3k+2t-4)+1,\text{a contradiction regarding}~|c(K_n)|.
\end{align*}
{\bf Scenario 2.3:}  $|c(K_n[S])|=3$, \(| c(K_n[S])- c(H)|=1 \), and  $|c(K_n[S])\cap c(P_2^a\cup P_2^b)|=2$, where $a,b\in[t]$. Assume \( c(s_1s_2) \notin c(H) \), $c(s_1s_3)=c(y^a_1y^a_2)$, and $c(s_2s_3)=c(y^b_1y^b_2)$. In this scenario, if any of the following conditions hold, we can construct a rainbow subgraph of \( K_n \) that is isomorphic to \( kP_3 \cup tP_2 \).
\begin{enumerate}
  \item There are $i\in[t]$ and $i_1\in[2]$ such that $s_3 y^i_{i_1} \in E(G)$ and $c(s_1s_2)\neq c(s_3 y^i_{i_1})$.
  \item There are $i\in([t]-\{a,b\})$, $i_1\in[2]$, $i_2\in([2]-\{i_1\})$, and $j\in\{a,b\}$ such that $\{y^i_{i_1}y^j_{1},y^i_{i_2}y^j_{2}\}\subseteq E(G)$ and $c(s_1s_2)\not\in c(\{y^i_{i_1}y^j_{1},y^i_{i_2}y^j_{2}\})$. 
  \item There are $i \in [k-1]$ and $j\in\{a,b\}$ such that for $A_1=\{y^j_{1}x^i_{2},y^j_{2}x^i_{2}, x^i_1x^i_3\}$, it holds that $A_1\subseteq E(G)$ and $c(s_1s_2) \notin c(A_1)$.
  \item There are $i \in [k-1]$, $j \in \{a,b\}$, $j_1\in[2]$,  $j_2\in\{1,3\}$, and $j_3\in([3]-\{j_2\})$ such that for $A_2=\{y^j_{j_1}x^i_{j_2}, s_3x^i_{j_3}\}$, it holds that $A_2\subseteq E(G)$ and $c(s_1s_2) \notin c(A_2)$.
  \item There are $i \in [k-1]$ and $j \in [3]$ such that for $A_3=\{y^a_1x^i_{2}, y^a_2x^i_{2},y^b_1x^i_{2}, y^b_2x^i_{2},s_jx^i_1,$ $s_jx^i_2\}$, it holds that $A_3\subseteq E(G)$ and $c(s_1s_2) \notin c(A_3)$.
  \item There are $i \in [k-1]$, $j\in\{a,b\}$, $j_1 \in [3]$, $j_2\in([3]-\{j_1\})$ such that for $A_4=\{y^j_1x^i_{j_1}, y^j_2x^i_{j_2}\}$, it holds that $A_4\subseteq E(G)$ and $c(s_1s_2) \notin c(A_4)$.
  \item It holds that either $\{y^a_1s_2,y^a_2s_2\}$ or $\{y^b_1s_1,y^b_2s_1\}$ is a subset of $E(G)$.
  \item There are $j\in\{a,b\}$, $i_1\in[2]$, and $i_2\in([2]-\{i_1\})$ such that $\{y^j_{i_1}s_1,y^j_{i_2}s_2\}$ $\subseteq$ $E(G)$.
  \item There are $i_1\in[2]$ and $i_2\in([2]-\{i_1\})$ such that $\{y^a_1y^b_{i_1},y^a_2y^b_{i_2}\}\subseteq E(G)$.
  \end{enumerate}
  Assuming that none of the previous statements is true, we obtain that $|\mathcal{S}_{2.3}(1,2)| \geq 2t + 4(t - 2)$, $|\mathcal{S}_{2.3}(3,\ldots,6)| \geq 11(k-1)$, and $|\mathcal{S}_{2.3}(7,8,9) \cup \{s_1s_3,s_2s_3\}| \geq 8$. From this and $k\geq2$, we can observe that
  \begin{align*}
  |c(K_n)|\leq & \frac{1}{2}(3k+2t)(3k+2t-1)-\big(2t+4(t-2)+11(k-1)+8\big)\\
  =& \frac{1}{2}(3k+2t-3)(3k+2t-4)+2-(2k-3)\\
  \leq &\frac{1}{2}(3k+2t-3)(3k+2t-4)+1,\text{a contradiction}.
\end{align*}
{\bf Scenario 2.4:} \(|c(K_n[S])|=3\) and \(|c(K_n[S]) - c(H)|=1\). Additionally, for \(a\in[t]\) and \(b\in[k-1]\), \(c(P_2^a) \in c(K_n[S])\) and \(|c(K_n[S]) \cap c(P_3^b)|=1\). We suppose that \(c(s_1s_2) \notin c(H)\), \(c(y_1^a y_2^a) = c(s_1s_3)\), and \(c(x_1^b x_2^b) = c(s_2s_3)\). Under this scenario, if one of the following holds, then we can build a rainbow subgraph of \(K_n\) that is isomorphic to \(kP_3 \cup tP_2\).
\begin{enumerate}
 \item There are $i\in[t]$, $i_1\in[2]$, and $z\in\{s_3,x^b_1\}$ such that $z y^i_{i_1} \in E(G)$ and $c(s_1s_2)\neq c(z y^i_{i_1})$.
 \item There are $i\in([t]-\{a\})$, $i_1\in[2]$, and $i_2\in([2]-\{i_1\})$ such that $\{y^i_{i_1}y^a_1,y^i_{i_2}y^a_2\}\subseteq E(G)$ and $c(s_1s_2)\not\in c(\{y^i_{i_1}y^a_1,y^i_{i_2}y^a_2\})$.
 \item For  $i\in ([k-1]-\{b\})$, there is a subset $B_1$ of $\{y^a_{a_1}x^i_{i_1}:a_1\in[2],i_1\in[3]\}\cup\{x^i_1x^i_3\}$ such that $|B_1|=3$, $B_1\subseteq E(G)$, and $c(s_1s_2)\notin c(B_1)$.
 \item For $i\in ([k-1]-\{b\})$, there is a subset $B_2$ of $\{s_{i_1}x^i_{i_2}:i_1,i_2\in[3]\}$ such that $|B_2|=7$, $B_2\subseteq E(G)$, and $c(s_1s_2)\notin c(B_2)$.
 \item For $i\in ([k-1]-\{b\})$, there is a subset $B_3$ of $\{x^b_{i_1}x^i_{i_2}:i_1,i_2\in[3]\}$ such that $|B_3|=7$, $B_3\subseteq E(G)$, and $c(s_1s_2)\notin c(B_3)$.
\item For $a_1\in[2]$ and $a_2\in([2]-\{a_1\})$, either $\{y^a_{a_1}s_1,y^a_{a_2}s_2\}\subseteq E(G)$ or $\{y^a_{1}s_2,y^a_{2}s_2\}\subseteq E(G)$ holds.
\item For $a_1\in[2]$ and $a_2\in([2]-\{a_1\})$, either $\big(\{y^a_{a_1}x^b_2,y^a_{a_2}x^b_3\}\subseteq E(G)$ and $c(s_1s_2)\notin c(\{y^a_{a_1}x^b_2,y^a_{a_2}x^b_3\})\big)$ or $\big(\{y^a_{1}x^b_3,y^a_{2}x^b_3\}\subseteq E(G)$ and $c(s_1s_2)\notin c(\{y^a_{1}x^b_3,y^a_{2}x^b_3\})\big)$ holds.
\item Both \( x_1^bx_3^b\in E(G) \) and $c(x_1^bx_3^b)\neq c(s_1s_2)$ hold.
\item For $i\in\{2,3\}$, it holds that $c(x^b_1x^b_3)=c(s_1s_i)$, $\{y^a_{1}x^b_2,y^a_{2}x^b_2\}\subseteq E(G)$, and $c(s_1s_2)$ $\notin$ $c(\{y^a_{1}x^b_2,$ $y^a_{2}x^b_2\})$.
\item It holds that $c(s_1s_3)=c(x^b_2x^b_3)$, $\{y^a_{1}x^b_2,y^a_{2}x^b_2\}\subseteq E(G)$, and $c(s_1s_2)$ $\notin$ $c(\{y^a_{1}x^b_2,$ $y^a_{2}x^b_2\})$.
\item There are $i\in([t]-\{a\})$, $i_1\in[2]$ and $i_2\in\{2,3\}$ such that for $B_4=\{y^a_{i_1}x^b_{i_2},y^a_1s_1,y^a_2s_1,y^i_1s_2,y^i_2s_2\}$, it holds that $B_4\subseteq E(G)$ and $c(x^b_1x^b_3)=c(s_2s_2)$.
\item There are $i\in([t]-\{a\})$ and $z\in\{s_1,x^b_2\}$ such that $c(x^b_1x^b_3)=c(y^i_1y^i_2)$, $\{zy^i_1,zy^i_2\}\subseteq E(G)$, and $c(s_1s_2)\notin c(\{zy^i_1,zy^i_2\})$. 
\item There are $i\in([t]-\{a\})$, $j\in ([k-1]-\{b\})$, $i_1\in[2]$, and $j_1\in\{1,3\}$ such that $c(x^b_1x^b_3)=c(x^j_{j_1}x^j_2)$, $y^i_{i_1}x^j_{j_1}\in E(G)$,  and $c(s_1s_2)\neq c(y^i_{i_1}x^j_{j_1})$.
\item There is a subset $B_5$ of $\{s_ix^b_j:i,j\in[3]\}$ such that $|B_5|=6$, $B_5\subseteq E(G)$, and $c(s_1s_2)\not\in c(B_5)$.
\end{enumerate}
Let's consider a situation where none of the above conditions are true. In this case, we can observe that $|\mathcal{S}_{2.4}(1,2)| \geq 4t + 2(t - 1)$, $|\mathcal{S}_{2.4}(3,4,5)| \geq 11(k-2)$, and $|\mathcal{S}_{2.4}(6,\ldots,14) \cup \{s_1s_3,s_2s_3\}| \geq 13$. From this, and since $k \geq 2$, we can conclude that
  \begin{align*}
  |c(K_n)|\leq & \frac{1}{2}(3k+2t)(3k+2t-1)-\big(4t+2(t-1)+11(k-1)+13\big)\\
  =& \frac{1}{2}(3k+2t-3)(3k+2t-4)+2-(2k-3)\\
  \leq &\frac{1}{2}(3k+2t-3)(3k+2t-4)+1,\text{a contradiction}.
\end{align*}
{\bf Scenario 2.5:}  Let \(|c(K_n[S])| = 3\) and \(|c(K_n[S]) - c(H)| = 1\). Additionally, for two distinct elements \(a\) and \(b\) in the set \([k-1]\), we have \(c(P_3^a) \cap c(K_n[S]) \neq \emptyset\) and \(c(P_3^b) \cap c(K_n[S]) \neq \emptyset\). We also assume that \(c(s_1s_2) \notin c(H)\), \(c(x_1^a x_2^a) = c(s_1s_3)\), and \(c(x_1^b x_2^b) = c(s_2s_3)\). 
Let \(Z = \{s_3, x_1^a, x_1^b\}\). Holding one of the following statements true will enable us, either directly or by employing one of the Scenarios 2.3 and 2.4, to create a rainbow subgraph of \(K_n\) that is isomorphic to \(kP_3 \cup tP_2\).
\begin{enumerate}
  \item There are $i\in[t]$, $i_1\in[2]$ and $z\in Z$ such that $y^i_{i_1}z\in E(G)$ and $c(s_1s_2)\neq c(y^i_{i_1}z)$.
  \item There is $i\in([k-1]-\{a,b\})$ such that $|\{x^i_{i_1}s_j:i_1,j\in[3]\}\cap E(G)|\geq7$.
  \item There are $i\in([k-1]-\{a,b\})$ and $j\in\{a,b\}$ such that $|\{x^i_{i_1}x^j_{j_1}:i_1,j_1\in[3]\}\cap E(G)|\geq7$.
  \item For $j\in\{a,b\}$, there is a subset $D_1$ of $\{s_ix^j_{j_1}:i,j_1\in[3]\}$ such that $|D_1|=6$, $D_1\subseteq E(G)$, and $c(s_1s_2)\not\in c(D_1)$.
  \item There is a subset $D_2$ of $\{x^a_{a_1}x^b_{b_1}:a_1,b_1\in[3]\}$ such that $|D_2|=6$, $D_2\subseteq E(G)$, and $c(s_1s_2)\not\in c(D_2)$.
  \item For \( j \in \{ a, b \} \), both \( x^j_1 x^j_3 \in E(G) \) and \( c(s_1 s_2) \neq c(x^j_1 x^j_3) \) occur.
   \item There are $i\in[t]$ and $i_1\in[2]$ such that either $\big(c(x^a_1x^a_3)=c(s_1s_3)$, $y^i_{i_1}x^a_2\in E(G)$ and $c(y^i_{i_1}x^a_2)\neq c(s_1s_2)\big)$ or $\big(c(x^b_1x^b_3)=c(s_2s_3)$, $y^i_{i_1}x^b_2\in E(G)$ and $c(y^i_{i_1}x^b_2)\neq c(s_1s_2)\big)$ holds.
 \item For $i\in\{a,b\}$, there are $j\in[t]$ and $j_1\in[2]$ such that $c(x^j_1x^j_3)=c(s_1s_2)$ and $y^i_{i_1}x^j_2\in E(G)$.
 \item There are $i\in[t]$, $i_1\in[2]$, and $i_2\in([2]-\{i_1\})$ such that $c(x^a_1x^a_3)=c(s_2s_3)$, $c(x^b_1x^b_3)=c(s_1s_3)$, $\{y^i_{i_1}x^a_2,y^i_{i_2}x^b_2\}\subseteq E(G)$, and $c(s_1s_2)\not\in c(\{y^i_{i_1}x^a_2,y^i_{i_2}x^b_2\})$.
 \item For $i\in\{a,b\}$, there are $j\in[t]$, $h\in([t]-\{j\})$, $h_1\in[2]$, and $h_2\in([2]-\{i_1\})$ such that either $\big(c(x^i_1x^i_3)=c(y^j_1y^j_2)$, $\{y^j_1x^i_2,y^j_2x^i_2\}\subseteq E(G)$, and $c(s_1s_2)\not\in c(\{y^j_1x^i_2,y^j_2x^i_2\})\big)$ or $\big(c(x^i_1x^i_3)=c(y^j_1y^j_2)$, $\{y^j_1y^h_{h_1},y^j_2y^h_{h_2}\}\subseteq E(G)$, and $c(s_1s_2)\not\in c(\{y^j_1y^h_{h_1},y^j_2y^h_{h_2}\})\big)$ holds. 
 \item For $i\in\{a,b\}$, there are $j\in([k-1]-\{a,b\})$, $j_1\in\{1,3\}$, $h\in[t]$, and $h_1\in[2]$ such that $c(x^i_1x^i_3)=c(x^j_{j_1}x^j_2)$, $y^h_{h_1}x^j_{j_1}\in E(G)$, and $c(s_1s_2)\neq c(y^h_{h_1}x^j_{j_1})$.
\end{enumerate}
Assuming that none of the above qualifications is true, we find that $|\mathcal{S}_{2.5}(1)| \geq 6t$, $|\mathcal{S}_{2.5}(2,3)| \geq 9(k-3)$, and $|\mathcal{S}_{2.5}(4,\ldots,11) \cup \{s_1s_3,s_2s_3\}| \geq 20$. From this, we can deduce that
 \begin{align*}
  |c(K_n)|\leq & \frac{1}{2}(3k+2t)(3k+2t-1)-\big(6t+9(k-3)+20\big)\\
  = &\frac{1}{2}(3k+2t-3)(3k+2t-4)+1,\text{a contradiction}.
\end{align*}
{\bf Scenario 2.6:} Let \(c(s_1s_2) \notin c(H)\), and \(c(s_1s_3) = c(s_2s_3) = c(x^a_1x^a_2)\), where \(a \in [k-1]\). Additionally, we assume that \(|c(K_n[V(P^a_3)])| = 3\). 
In this case, we first redefine the set \(S\) and the path \(P_3^a\) as follows: let \(S = \{x^a_i : i \in [3]\}\), \(V(P^a_3) = \{s_i : i \in [3]\}\), and \(E(P_3^a) = \{s_1s_2, s_2s_3\}\). After this redefinition, we can apply one of the scenarios from 2.1 to 2.5 to obtain the desired result.\\
{\bf Scenario 2.7:} Let \(c(s_1s_2) \notin c(H)\), and \(c(s_1s_3) = c(s_2s_3) = c(x^a_1x^a_2)\), where \(a \in [k-1]\). Plus, $c(s_3x^a_1)\not\in (c(H)-\{c(x^a_1x^a_2)\})$ and \(|c(K_n[V(P^a_3)])| =2\).
 In this case, if one of the following holds, then we can construct a rainbow subgraph of \(K_n\) that is isomorphic to \(kP_3 \cup tP_2\). 
 \begin{enumerate}
    \item For $z\in\{s_3,x^a_1\}$, there are $i\in[t]$ and $j\in[2]$ such that $zy^i_j\in E(G)$ and $c(s_1s_2)\neq c(zy^i_j)$.
    \item There are $i\in[t]$, $i_1\in[2]$, $i_2\in([2]-\{i_1\})$, $z_1\in\{s_1,s_2\}$, and $z_2\in\{x^a_2,x^a_3\}$ such that $\{y^i_{i_1}z_1,y^i_{i_2}z_2\}\subseteq E(G)$ and $c(s_1s_2),c(s_3x^a_1)\not\in c(\{y^i_{i_1}z_1,y^i_{i_2}z_2\})$.
    \item There are $i\in[t]$, $i_1\in[2]$, $j\in([k-1]-\{a\})$, $j_1\in\{1,3\}$, $z\in\{s_3,x_1^a\}$, and $j_2\in(\{1,3\}-\{j_1\})$ such that $\{zx^j_{j_1},x^j_{j_2}y^i_{i_1}\}\subseteq E(G)$ and $c(s_1s_2)\not\in c(\{zx^j_{j_1},x^j_{j_2}y^i_{i_1}\})$.
    \item There are $i\in[t]$, $i_1\in[2]$, $j\in([k-1]-\{a\})$, $j_1\in\{1,3\}$, and $z\in\{s_3,x_1^a\}$ such that $\{zx^j_{2},x^j_{j_1}y^i_{i_1}\}\subseteq E(G)$ and $c(s_1s_2)\not\in c(\{zx^j_{2},x^j_{j_1}y^i_{i_1}\})$.
     \item There are $i\in[t]$, $i_1\in[2]$, $j\in([k-1]-\{a\})$, and $z\in\{s_3,x_1^a\}$ such that $\{zx^j_{1},zx^j_3,x^j_{2}y^i_{i_1}\}\subseteq E(G)$ and $c(s_1s_2)\not\in c(\{zx^j_{1},zx^j_3,x^j_{2}y^i_{i_1}\})$.
      \item For two distinct elements $j_1,j_2\in[3]$, there are $i\in[t]$, $i_1\in[2]$, \( j \in ([k-1] - \{a\}) \),  $z_1\in\{s_1,s_2\}$, $z_2\in\{x^a_2,x^a_3\}$, and $j_3\in([3]-\{j_1,j_2\})$ such that $\{z_1x^j_{j_1},z_2x^j_{j_2}, x^j_{j_3}y^i_{i_1}\}\subseteq E(G)$ and $c(s_1s_2)\not\in c(\{z_1x^j_{j_1},z_2x^j_{j_2}, x^j_{j_3}y^i_{i_1}\})$.
      \item Both $ x_1^ax_3^a\in E(G)$ and $c(s_1s_2)\neq c(x_1^ax_3^a)$ hold. 
       \item There is a subset $F$ of $\{s_ix^a_{a_1}:i,a_1\in[3]\}$ such that $|F|=6$, $F\subseteq E(G)$, and $c(s_1s_2)\not\in c(F)$.
 \end{enumerate}
Assuming none of the previously mentioned conditions is true, we conclude that \( |\mathcal{S}_{2.7}(1,2)| \geq 8t \), \( |\mathcal{S}_{2.7}(3,\ldots,6)| \geq 12(k-2) \), and \( |\mathcal{S}_{2.7}(7,8) \cup \{s_1s_3,s_2s_3\}| \geq 7 \). From this, and $k,t\geq2$, we can find that 
\begin{align*}
  |c(K_n)|\leq & \frac{1}{2}(3k+2t)(3k+2t-1)-(8t+12(k-2)+7)\\
  =& \frac{1}{2}(3k+2t-3)(3k+2t-4)+2-(3k+2t-9)\\
  \leq &\frac{1}{2}(3k+2t-3)(3k+2t-4)+1, \text{a contradiction}.
\end{align*}
 {\bf Scenario 2.8:} Let \(c(s_1s_2) \notin c(H)\), and \(c(s_1s_3) = c(s_2s_3) = c(x^a_1x^a_2)\), where \(a \in [k-1]\). Plus, $c(s_3x^a_1)=c(s_1s_2)$ and \(|c(K_n[V(P^a_3)])| =2\).
 In this case, if one of the following holds, then we can construct a rainbow subgraph of \(K_n\) that is isomorphic to \(kP_3 \cup tP_2\). 
 \begin{enumerate}
   \item There are $i\in[t]$, $j\in[2]$, and $z\in\{s_1,s_2,s_3,x^a_1\}$ such that $zy^i_j\in E(G)$ and $c(s_1s_2)\neq c(zy^i_j)$. 
    \item There are $i\in[t]$, $i_1\in[2]$, $j\in([k-1]-\{a\})$, $j_1\in\{1,3\}$, $z\in\{s_3,x_1^a\}$, and $j_2\in(\{1,3\}-\{j_1\})$ such that $\{zx^j_{j_1},x^j_{j_2}y^i_{i_1}\}\subseteq E(G)$ and $c(s_1s_2)\not\in c(\{zx^j_{j_1},x^j_{j_2}y^i_{i_1}\})$.
   \item There are $i\in[t]$, $i_1\in[2]$, $j\in([k-1]-\{a\})$, $j_1\in\{1,3\}$, and $z\in\{s_3,x_1^a\}$ such that $\{zx^j_{2},x^j_{j_1}y^i_{i_1}\}\subseteq E(G)$ and $c(s_1s_2)\not\in c(\{zx^j_{2},x^j_{j_1}y^i_{i_1}\})$.
   \item There are $i\in[t]$, $i_1\in[2]$, $j\in([k-1]-\{a\})$, and $z\in\{s_3,x_1^a\}$ such that $\{zx^j_{1},zx^j_3,x^j_{2}y^i_{i_1}\}\subseteq E(G)$ and  $c(s_1s_2)\not\in c(\{zx^j_{1},zx^j_3,x^j_{2}y^i_{i_1}\})$. 
   \item There are  $j\in([k-1]-\{a\})$, $j_1\in\{1,3\}$, and  $j_2\in([3]-\{j_1\})$ such that $\{s_1x^j_{j_1},s_2x^j_{j_1},x^a_3x^j_{j_2}\}\subseteq E(G)$ and $c(s_1s_2)\not\in c(\{s_1x^j_{j_1},s_2x^j_{j_1},x^a_3x^j_{j_2}\})$.
   \item There are $j\in([k-1]-\{a\})$ and $j_1\in\{1,3\}$ such that $\{s_1x^j_{2},s_2x^j_{2},x^a_3x^j_{j_1},x_1^jx_3^j\}\subseteq E(G)$ and  $c(s_1s_2)\not\in c(\{s_1x^j_{2},s_2x^j_{2},x^a_3x^j_{j_1},x_1^jx_3^j\})$.
   \item Both \( x_1^ax_3^a\in E(G) \) and $c(s_1s_2)\neq c(x_1^ax_3^a)$ hold.
   \item There are a subset $F_1$ of $\{s_ix^a_j:i\in[2],j\in[3]\}\cup\{s_3x^a_j:j\in\{2,3\}\}$ such that $F_1\subseteq E(G)$, $|F_1|=6$, and $c(s_1s_2)\not\in c(F_1)$.
 \end{enumerate}
Let's assume that none of the statements mentioned above are true. We can observe that \( |\mathcal{S}_{2.8}(1)| \geq 8t \), \( |\mathcal{S}_{2.8}(2,\ldots,6)| \geq 9(k-2) \), and \( |\mathcal{S}_{2.8}(7,8) \cup \{s_1s_3,s_2s_3\}| \geq 7 \). Given that \( t \geq 2 \), we can derive that
\begin{align*}
  |c(K_n)|\leq & \frac{1}{2}(3k+2t)(3k+2t-1)-(8t+9(k-2)+7)\\
  =& \frac{1}{2}(3k+2t-3)(3k+2t-4)+2-(2t-3)\\
  \leq &\frac{1}{2}(3k+2t-3)(3k+2t-4)+1, \text{a contradiction}.
\end{align*} 
{\bf Scenario 2.9:} Let \(c(s_1s_2) \notin c(H)\), and \(c(s_1s_3) = c(s_2s_3) = c(x^a_1x^a_2)\), where \(a \in [k-1]\). Plus, $c(s_3x^a_1)=c(x^a_2x^a_3)$ and \(|c(K_n[V(P^a_3)])| =2\).
 In this scenario, possessing one of the following will enable us to construct a rainbow subgraph of \(K_n\) that is isomorphic to \(kP_3 \cup tP_2\). 
 \begin{enumerate}
   \item There are $i\in[t]$, $j\in[2]$, and $z\in\{s_3,x^a_1,x^a_2,x^a_3\}$ such that $zy^i_j\in E(G)$ and $c(s_1s_2)\neq c(zy^i_j)$. 
   \item There are $i\in[t]$, $i_1\in[2]$, $j\in([k-1]-\{a\})$, $j_1\in\{1,3\}$, $z\in\{s_3,x_1^a\}$, and $j_2\in(\{1,3\}-\{j_1\})$ such that $\{zx^j_{j_1},x^j_{j_2}y^i_{i_1}\}\subseteq E(G)$ and $c(s_1s_2)\not\in c(\{zx^j_{j_1},x^j_{j_2}y^i_{i_1}\})$.
   \item There are $i\in[t]$, $i_1\in[2]$, $j\in([k-1]-\{a\})$, $j_1\in\{1,3\}$, and $z\in\{s_3,x_1^a\}$ such that $\{zx^j_{2},x^j_{j_1}y^i_{i_1}\}\subseteq E(G)$ and $c(s_1s_2)\not\in c(\{zx^j_{2},x^j_{j_1}y^i_{i_1}\})$.
   \item There are $i\in[t]$, $i_1\in[2]$, $j\in([k-1]-\{a\})$, and $z\in\{s_3,x_1^a\}$ such that $\{zx^j_{1},zx^j_3,x^j_{2}y^i_{i_1}\}\subseteq E(G)$ and $c(s_1s_2)\not\in c(\{zx^j_{1},zx^j_3,x^j_{2}y^i_{i_1}\})$.
    \item There are  $j\in([k-1]-\{a\})$, $j_1\in\{1,3\}$, and $j_2\in([3]-\{j_1\})$ such that $\{x^a_2x^j_{j_1},x^a_3x^j_{j_1},s_1x^j_{j_2}\}\subseteq E(G)$. 
   \item There are $j\in([k-1]-\{a\})$ and $j_1\in\{1,3\}$ such that $\{x^a_2x^j_{2},x^a_3x^j_{2},s_1x^j_{j_1},x_1^jx_3^j\}\subseteq E(G)$.
    \item Both \( x_1^ax_3^a\in E(G) \) and $c(s_1s_2)\neq c(x_1^ax_3^a)$ hold.
   \item There are a subset $F_2$ of $\{s_ix^a_j:i\in[2],j\in[3]\}\cup\{s_3x^a_j:j\in\{2,3\}\}$ such that $F_2\subseteq E(G)$, $|F_2|=6$, and $c(s_1s_2)\not\in c(F_1)$.
 \end{enumerate}
 Let's consider that none of the conditions mentioned above are true. We can see that \( |\mathcal{S}_{2.9}(1)| \geq 8t \), \( |\mathcal{S}_{2.9}(2,\ldots,6)| \geq 9(k-2) \), and \( |\mathcal{S}_{2.9}(7,8) \cup \{s_1s_3,s_2s_3\}| \geq 7 \). From this, and \( t \geq 2 \), we can find that
\begin{align*}
  |c(K_n)|\leq & \frac{1}{2}(3k+2t)(3k+2t-1)-(8t+9(k-2)+7)\\
  =& \frac{1}{2}(3k+2t-3)(3k+2t-4)+2-(2t-3)\\
  \leq &\frac{1}{2}(3k+2t-3)(3k+2t-4)+1, \text{a contradiction}.
\end{align*} 
 {\bf Scenario 2.10:} Let \(c(s_1s_2) \notin c(H)\) and \(c(s_1s_3) = c(s_2s_3) = c(x^a_1x^a_2)\), where \(a \in [k-1]\). Additionally, \(c(s_3x^a_1) = c(y^l_1y^l_2)\), where \(l \in [t]\).
In this case, if any of the following conditions hold, we can either construct a rainbow subgraph of \(K_n\) that is isomorphic to \(kP_3 \cup tP_2\) or redefine the set \(S\) and the graph \(H\). We can then apply one of the scenarios from sections 2.6 to 2.9 to achieve the desired outcome.
 \begin{enumerate}
   \item There are $i\in[t]$, $j\in[2]$, and $z\in\{s_3,x^a_1\}$ such that $zy^i_j\in E(G)$ and $c(s_1s_2)\neq c(zy^i_j)$. 
    \item There is $i\in([t]-\{l\})$, $i_1\in[2]$, and $i_2\in([2]-\{i_1\})$ such that $\{y^i_{i_1}y^l_{1},y^i_{i_2}y^l_{2}\}\subseteq E(G)$ and $c(s_1s_2)\not\in c(\{y^i_{i_1}y^l_{1},y^i_{i_2}y^l_{2}\})$.
   \item There are $i\in[t]$, $i_1\in[2]$, $j\in([k-1]-\{a\})$, $j_1\in\{1,3\}$, $z\in\{s_3,x_1^a\}$, and $j_2\in(\{1,3\}-\{j_1\})$ such that $\{zx^j_{j_1},x^j_{j_2}y^i_{i_1}\}\subseteq E(G)$ and $c(s_1s_2)\not\in c(\{zx^j_{j_1},x^j_{j_2}y^i_{i_1}\})$.
   \item There are $i\in[t]$, $i_1\in[2]$, $j\in([k-1]-\{a\})$, $j_1\in\{1,3\}$, and $z\in\{s_3,x_1^a\}$ such that $\{zx^j_{2},x^j_{j_1}y^i_{i_1}\}\subseteq E(G)$ and $c(s_1s_2)\not\in c(\{zx^j_{2},x^j_{j_1}y^i_{i_1}\})$.
   \item There are $i\in[t]$, $i_1\in[2]$, $j\in([k-1]-\{a\})$, and $z\in\{s_3,x_1^a\}$ such that $\{zx^j_{1},zx^j_3,x^j_{2}y^i_{i_1}\}\subseteq E(G)$ and $c(s_1s_2)\not\in c(\{zx^j_{1},zx^j_3,x^j_{2}y^i_{i_1}\})$.
   \item There are $j\in([k-1]-\{a\})$, $j_1\in[3]$, and $j_2\in([3]-\{j_1\})$ such that $\{y^l_1x^j_{j_1},y^l_2x^j_{j_2}\}\subseteq E(G)$ and $c(s_1s_2)\not\in c(\{y^l_1x^j_{j_1},y^l_2x^j_{j_2}\})$.
  \item There are $j\in([k-1]-\{a\})$ and $j_1\in\{1,3\}$ such that $\{y^l_1x^j_{j_1},y^l_2x^j_{j_1}\}\subseteq E(G)$ and $c(s_1s_2)\not\in c(\{y^l_1x^j_{j_1},y^l_2x^j_{j_1}\})$.
  \item There are $i\in[t]$ and $j\in[2]$ such that $\{s_3x^a_{2},x^a_3y^i_{j}\}\subseteq E(G)$ and $c(s_1s_2)\not\in c(\{s_3x^a_{2},x^a_3y^i_{j}\})$.
  \item There are $i\in[2]$ such that $\{s_3x^a_{2},s_ix^a_3\}\subseteq E(G)$ and $c(s_1s_2)\not\in c(\{s_3x^a_{2},s_ix^a_3\})$.
  \item There are  \(i\in[t]\), \(i_1,j\in[2]\), and \(i_2\in([2]-\{i_1\})\) such that either \(\big(\{s_3x^a_3,s_jy^i_{i_1},$ $x^a_2y^i_{i_2}\}\subseteq E(G)\) and $c(s_1s_2)\not\in c(\{s_3x^a_3,s_jy^i_{i_1},x^a_2y^i_{i_2}\})\big)$, \(\big(\{s_3x^a_3,s_jy^i_{i_1},x^a_3y^i_{i_2}\}\subseteq E(G)\) and $c(s_1s_2)\not\in c(\{s_3x^a_3,s_jy^i_{i_1},x^a_3y^i_{i_2}\})\big)$, or \(\big(\{s_3x^a_3,x^a_2y^i_{i_1},x^a_3y^i_{i_2}\}\subseteq E(G)\) and  $c(s_1s_2)\not\in c(\{s_3x^a_3,x^a_2y^i_{i_1},x^a_3y^i_{i_2}\})\big)$ holds.
  \item There are $i\in[t]$, $i_1,j_1\in[2]$, $i_2\in([2]-\{i_1\})$, and $j_2\in([3]-\{1\})$ such that either $\{x^a_1s_{j_1},y^i_{i_1}s_{1},y^i_{i_2}x^a_{j_2}\}$, $\{x^a_1s_{j_1},y^i_{i_1}s_{2},y^i_{i_2}x^a_{j_2}\}$, or $\{x^a_1s_{j_1},y^i_{i_1}s_{1},y^i_{i_2}s_2\}$ holds.
   \item There are $i_1\in[2]$, $i_2\in([2]-\{i_1\})$, $j_1\in[2]$, and $j_2\in\{2,3\}$ such that $\{y^l_{i_1}s_{j_1},y^l_{i_2}x^a_{j_2}\}\subseteq E(G)$ and $c(s_1s_2)\not\in c(\{y^l_{i_1}s_{j_1},y^l_{i_2}x^a_{j_2}\})$.
    \item There are $i\in([t]-\{l\})$, $i_1\in[2]$, $i_2\in([2]-\{i_1\})$, and $z\in\{s_1,s_2,x^a_2,x^a_3\}$ such that $\{y^l_1y^i_{i_1},y^l_2y^i_{i_1},y^i_{i_2}z\}\subseteq E(G)$ and $c(s_1s_2)\not\in c(\{y^l_1y^i_{i_1},y^l_2y^i_{i_1},y^i_{i_2}z\})$.
    \item There are $i\in([t]-\{l\})$, $i_1,l_1\in[2]$, $l_2\in([2]-\{l_1\})$, and $z\in\{s_1,s_2,x^a_2,x^a_3\}$ such that $\{y^l_{l_1}y^i_{i_1},y^l_{l_2}z\}\subseteq E(G)$ and $c(s_1s_2)\not\in c(\{y^l_{l_1}y^i_{i_1},y^l_{l_2}z\})$.
 \end{enumerate}
Assuming that none of the above conditions are true, we calculate that \( |\mathcal{S}_{2.10}(1,2)| \geq 4t + 2(t - 1) \), \( |\mathcal{S}_{2.10}(3,\ldots,7)| \geq 9(k - 2) \), and \( |\mathcal{S}_{2.10}(8,\ldots,14) \cup \{s_1s_3, s_2s_3\}| \geq 13 \). From this, we can conclude that
  \begin{align*}
  |c(K_n)|\leq & \frac{1}{2}(3k+2t)(3k+2t-1)-(4t+2(t-1)+9(k-2)+13)\\
  = &\frac{1}{2}(3k+2t-3)(3k+2t-4)+1, \text{a contradiction}.
\end{align*}
 {\bf Scenario 2.11:} Let \(c(s_1s_2) \notin c(H)\), and let \(c(s_1s_3) = c(s_2s_3) = c(x^a_1x^a_2)\), where \(a \in [k-1]\). Additionally, we have \(c(s_3x^a_1) = c(x^r_2x^r_3)\), where \(r \in ([k-1] - \{a\})\).
Under these conditions, if we hold one of the following assumptions, we can either construct a rainbow subgraph of \(K_n\) that is isomorphic to \(kP_3 \cup tP_2\) or redefine the set \(S\) and the graph \(H\). We can then apply one of the scenarios from Sections 2.6 to 2.10 to achieve the desired result.
 \begin{enumerate}
   \item There are $i\in[t]$, $j\in[2]$, and $z\in\{s_3,x^a_1,x_1^r,x_3^r\}$ such that $zy^i_j\in E(G)$ and $c(s_1s_2)\neq c(zy^i_j)$. 
    \item There are $i\in[t]$, $i_1\in[2]$, $j\in([k-1]-\{a,r\})$, $j_1\in\{1,3\}$, $z\in\{s_3,x_1^a,x^r_1,x^r_3\}$, and $j_2\in(\{1,3\}-\{j_1\})$ such that $\{zx^j_{j_1},x^j_{j_2}y^i_{i_1}\}\subseteq E(G)$ and $c(s_1s_2)\not\in c(\{zx^j_{j_1},x^j_{j_2}y^i_{i_1}\})$.
   \item There are $i\in[t]$, $i_1\in[2]$, $j\in([k-1]-\{a,r\})$, $j_1\in\{1,3\}$, and $z\in\{s_3,x_1^a,x^r_1,x^r_3\}$ such that $\{zx^j_{2},x^j_{j_1}y^i_{i_1}\}\subseteq E(G)$ and $c(s_1s_2)\not\in c(\{zx^j_{2},x^j_{j_1}y^i_{i_1}\})$.
   \item There are $i\in[t]$, $i_1\in[2]$,  $j\in([k-1]-\{a,r\})$, and $z\in\{s_3,x_1^a,x^r_1,x^r_3\}$ such that $\{zx^j_{1},zx^j_3,x^j_{2}y^i_{i_1}\}\subseteq E(G)$ and $c(s_1s_2)\not\in c(\{zx^j_{1},zx^j_3,x^j_{2}y^i_{i_1}\})$.
   \item There is $i\in\{a,r\}$ such that $|\{x^i_{i_1}s_j:i_1,j\in[3]\}\cap E(G)|\geq7$.
   \item There is a subset $A_1$ of $\{x^a_{a_1}x^r_{r_1}:a_1,r_1\in[3]\}$ such that $A_1\subseteq E(G)$, $|A_1|=7$, and $c(s_1s_2)\not\in c(A_1)$.
    \item There are $i\in[t]$ and $j\in[2]$ such that $\{s_3x^a_{2},x^a_3y^i_{j}\}\subseteq E(G)$ and $c(s_1s_2)\not\in c(\{s_3x^a_{2},x^a_3y^i_{j}\})$.
    \item There is $i\in[2]$ such that $\{s_3x^a_{2},s_ix^a_3\}\subseteq E(G)$ and $c(s_1s_2)\not\in c(\{s_3x^a_{2},s_ix^a_3\})$.
    \item There are \(i\in[t]\), \(i_1\in[2]\), $i_2\in([2]-\{i_1\})$,  and \(j\in[2]\) such that  either \(\big(\{s_3x^a_3,s_jy^i_{i_1},x^a_2y^i_{i_2}\}\subseteq E(G)\) and $c(s_1s_2)\not\in c(\{s_3x^a_3,s_jy^i_{i_1},x^a_2y^i_{i_2}\})\big)$, \(\big(\{s_3x^a_3,$ $s_jy^i_{i_1},$ $x^a_3y^i_{i_2}\}$ $\subseteq E(G)\) and $c(s_1s_2)\not\in c(\{s_3x^a_3,s_jy^i_{i_1},x^a_3y^i_{i_2})\big)$, or \(\big(\{s_3x^a_3,$ $x^a_2y^i_{i_1},$ $x^a_3y^i_{i_2}\}$ $\subseteq E(G)\) and $c(s_1s_2)\not\in c(\{s_3x^a_3,x^a_2y^i_{i_1},x^a_3y^i_{i_2}\})\big)$ holds.
    \item There are $i\in[t]$, $i_1\in[2]$, $i_2\in([2]-\{i_1\})$, $j_1\in[2]$, and $j_2\in([3]-\{1\})$ such that either $\{x^a_1s_{j_1},y^i_{i_1}s_{1},y^i_{i_2}x^a_{j_2}\}$, $\{x^a_1s_{j_1},y^i_{i_1}s_{2},y^i_{i_2}x^a_{j_2}\}$, or $\{x^a_1s_{j_1},y^i_{i_1}s_{1},y^i_{i_2}s_2\}$ is a subset of $E(G)$.
     \item There are $i\in[t]$, $i_1\in[2]$, $i_2\in([2]-\{i_1\})$, $j_1\in\{1,3\}$, and $j_2\in[2]$ such that  $\{s_3x^r_{j_1},s_{j_2}y^i_{i_1},x^r_2y^i_{i_2}\}\subseteq E(G)$ and $c(s_1s_2)\not\in c(\{s_3x^r_{j_1},s_{j_2}y^i_{i_1},x^r_2y^i_{i_2}\})$.
      \item There are $i\in[t]$, $i_1\in[2]$, and $j\in[3]$ such that  $\{s_jx^r_1,s_jx^r_3,y^i_{i_1}x^r_2\}\subseteq E(G)$ and $c(s_1s_2)\not\in c(\{s_jx^r_1,s_jx^r_3,y^i_{i_1}x^r_2\})$.
      \item There are $i\in[3]$ and $j\in\{2,3\}$ such that $\{s_ix^r_1,s_ix^r_3,x^a_jx^r_2\}\subseteq E(G)$ and $c(s_1s_2)\not\in c(\{s_ix^r_1,s_ix^r_3,x^a_jx^r_2\})$.
      \item There is $i\in\{a,r\}$ such that $x^i_1x^i_3\in E(G)$ and $c(s_1s_2)\neq c(x^i_1x^i_3)$.
 \end{enumerate}
 Let’s assume that none of the above statements are true. We can observe that \( |\mathcal{S}_{2.11}(1)| \geq 8t \), \( |\mathcal{S}_{2.11}(2,3,4)| \geq 12(k - 3) \), and \( |\mathcal{S}_{2.11}(5,\ldots,14) \cup \{s_1s_3, s_2s_3\}| \geq 16 \). Since \( t \geq 2 \) and \( k \geq 3 \), it follows that
  \begin{align*}
  |c(K_n)|\leq & \frac{1}{2}(3k+2t)(3k+2t-1)-(8t+9(k-3)+16)\\
  =&\frac{1}{2}(3k+2t-3)(3k+2t-4)+2-(3k+2t-12)\\
  \leq &\frac{1}{2}(3k+2t-3)(3k+2t-4)+1,\text{a contradiction}.
\end{align*}
 {\bf Scenario 2.12:} Let \(c_1=c(s_1s_2)=c(s_1s_3) \notin c(H)\), and \( c(s_2s_3) = c(y^l_1y^l_2)\), where \(l \in [t]\).
 In this case, if one of the following holds, then we can construct a rainbow subgraph of \(K_n\) that is isomorphic to \(kP_3 \cup tP_2\).
 \begin{enumerate}
 \item There are $i\in[t]$, $i_1\in[2]$, and $j\in\{2,3\}$ such that $s_jy^i_{i_1}\in E(G)$ and $c(s_jy^i_{i_1})\neq c_1$.
 \item There are $i\in([t]-\{l\})$, $i_1\in[2]$, and $i_2\in([2]-\{i_1\})$ such that $\{y^i_{i_1}y^l_1,y^i_{i_2}y^l_2\}\subseteq E(G)$ and $c_1\not\in c(\{y^i_{i_1}y^l_1,y^i_{i_2}y^l_2\})$.
  \item For $i\in[k-1]$, there is a subset $B_1$ of $\{s_{i_1}x^i_{i_2}:i_1,i_2\in[3]\}$ such that $B_1\subseteq E(G)$, $|B_1|=4$, and $c_1\not\in c(B_1)$.
   \item For $i\in[k-1]$, there is a subset $B_2$ of $\{y^l_{l_1}x^i_{i_1}:l_1\in[2],i_1\in[3]\}$ such that $B_2\subseteq E(G)$, $|B_2|=4$, and $c_1\not\in c(B_2)$.
   \item There are $i\in([t]-\{l\})$, $l_1,i_1\in[2]$, and $l_2\in([2]-\{l_1\})$ such that $\{s_1y^l_{l_1},y^l_{l_2}y^i_{i_1}\}$ $\subseteq E(G)$.
   \item There are $i\in([t]-\{l\})$ and $l_1\in[2]$ such that $\{s_1y^l_{l_1},s_1y^i_{1},s_1y^i_2,y^l_{l_1}y^i_1,y^l_{l_1}y^i_2\}$ $\subseteq E(G)$.
    \item It holds that $\{s_1y^l_1,s_1y^l_2\}\subseteq E(G)$. 
 \end{enumerate}
 If none of the above conditions are true, we can see that \( |\mathcal{S}_{2.12}(1,2)| \geq 4t + 2(t - 1) \), \( |\mathcal{S}_{2.12}(3,4)| \geq 9(k - 1) \), and \( |\mathcal{S}_{2.12}(5,\ldots,14) \cup \{s_2s_3\}| \geq 4 \). Thus, we have
  \begin{align*}
  |c(K_n)|\leq & \frac{1}{2}(3k+2t)(3k+2t-1)-(4t+2(t-1)+9(k-1)+4)\\
  =&\frac{1}{2}(3k+2t-3)(3k+2t-4)+1, \text{a contradiction}.
\end{align*}
{\bf Scenario 2.13:} Let \(c_1=c(s_1s_2)=c(s_1s_3) \notin c(H)\), and \( c(s_2s_3) = c(x^r_1x^r_2)\), where \(r \in [k-1]\).
 In this case, if one of the following holds, then we can either construct a rainbow subgraph of \(K_n\) that is isomorphic to \(kP_3 \cup tP_2\) or redefine the set \(S\) and the graph \(H\). We can then apply Section 2.12 to achieve the desired result.
  \begin{enumerate}
 \item There are $i\in[t]$, $i_1\in[2]$, and $z\in\{s_2,s_3,x^r_1,x^r_3\}$ such that $zy^i_{i_1}\in E(G)$ and $c_1\neq c(zy^i_{i_1})$.
 \item For $i\in([k-1]-\{r\})$, there is a subset $F_1$ of $\{s_{i_1}x^i_{i_2}:i_1,i_2\in[3]\}$ such that $F_1\subseteq E(G)$, $|F_1|=4$, and $c_1\not\in c(F_1)$.
  \item For $i\in([k-1]-\{r\})$, there is a subset $F_2$ of $\{x^r_{r_1}x^i_{i_1}:r_1,i_1\in[3]\}$ such that $F_2\subseteq E(G)$, $|F_2|=7$, and $c_1\not\in c(F_2)$.
  \item There is a subset $F_3$ of $\{s_{i_1}x^r_{r_1}:i_1,r_1\in[3]\}$ such that $F_3\subseteq E(G)$, $|F_3|=6$, and $c_1\not\in c(F_3)$.
  \item Both $x_1^rx_3^r\in E(G)$ and $c_1\neq c(x_1^rx_3^r)$ happen.
 \end{enumerate}
 Assuming that none of the previous statements are true, we calculate that \( |\mathcal{S}_{2.13}(1)| \geq 8t \), \( |\mathcal{S}_{2.13}(2,3)| \geq 9(k - 2) \), and \( |\mathcal{S}_{2.13}(4,5) \cup \{s_2s_3\}| \geq 7 \). Given that \( t \geq 2 \), we derive that
  \begin{align*}
  |c(K_n)|\leq & \frac{1}{2}(3k+2t)(3k+2t-1)-(8t+9(k-2)+7)\\
  =&\frac{1}{2}(3k+2t-3)(3k+2t-4)+2-(2t-3)\\
  \leq& \frac{1}{2}(3k+2t-3)(3k+2t-4)+1,\text{a contradiction}.
\end{align*}
{\bf Scenario 2.14:} Let \(c_1 = c(s_1s_2) = c(s_1s_3) = c(s_2s_3) \notin c(H)\). In this case, if any of the following conditions occur, then we can construct a rainbow subgraph of \(K_n\) that is isomorphic to \(kP_3 \cup tP_2\).
  \begin{enumerate}
 \item There are $i\in[t]$, $i_1\in[2]$, and $j\in[3]$ such that $y^i_{i_1}s_j\in E(G)$ and $c_1\neq c(y^i_{i_1}s_j)$.
 \item For $i\in[k-1]$, there is a subset $A_1$ of $\{s_{i_1}x^i_{i_2}:i_1,i_2\in[3]\}\cup \{x^i_1x^i_3\}$ such that $A_1\subseteq E(G)$ and $|A_1|=2$, and $c_1\not\in c(A_1)$.
 \end{enumerate}
 If neither of the two conditions above holds, we can see that \( |\mathcal{S}_{2.14}(1)| \geq 6t \) and \( |\mathcal{S}_{2.14}(2)| \geq 9(k - 1) \). Therefore, since \( |E(G)\cap\{s_1s_2,s_1s_3,s_2s_3\}|\leq 1 \), it follows that
  \begin{align*}
  |c(K_n)|\leq & \frac{1}{2}(3k+2t)(3k+2t-1)-(6t+9(k-1)+2)\\
  =&\frac{1}{2}(3k+2t-3)(3k+2t-4)+1, \text{a contradiction}.
\end{align*}
{\bf Scenario 2.15:} Let \( c(K_n[S]) \subseteq c(H) \). Moreover, suppose that \( c(x^a_1x^a_2) \in c(K_n[S]) \), where \( a \in [k-1] \). Under these conditions, if any of the following statements are true, we can either construct a rainbow subgraph of \( K_n \) that is isomorphic to \( kP_3 \cup tP_2 \), or we can apply one of the scenarios from 2.1 to 2.14 to achieve the desired result.
\begin{enumerate}
  \item There are $i\in[t]$, $j_1\in[3]$, and $j_2\in([3]-\{j_1\})$ such that $\{y^i_1s_{j_1},y^i_2s_{j_2}\}\subseteq E(G)$.
  \item There are $i\in[t]$ and $i_1\in[2]$ such that  $|\{y^i_{i_1}s_i:i\in[3]\}\cap E(G)|=3$.
   \item There are $i\in[t]$ and $i_1\in[2]$ such that $x^a_{1}y^i_{i_1}\in E(G)$.
  \item For each \( i \in [k-1] \), there are two distinct elements \( i_1, i_2 \in [3] \), as well as two other distinct elements \( j_1, j_2 \in [3] \), not necessarily different from $i_1$ and $i_2$,   such that \( \{ x^i_{i_1} s_{j_1}, x^i_{i_2} s_{j_2} \} \subseteq E(G) \).
    \item There is $i\in([k-1]-\{a\})$ such that $|\{x^i_{i_1}x^a_{a_1}:i_1,a_1\in[3]\}\cap E(G)|\geq7$.
  \item There is $i\in[3]$ such that $|\{x^a_{i}s_j:j\in[3]\}\cap E(G)|=3$.
  \item It holds that $x^a_1x^a_3\in E(G)$.
\end{enumerate}
Let’s assume that none of the previous statements are true. Thus, we can conclude that \( |\mathcal{S}_{2.15}(1,2,3)| \geq 6t \), \( |\mathcal{S}_{2.15}(4,5)| \geq 6(k - 1) + 3(k - 2) \), and \( |\mathcal{S}_{2.15}(6,7)\cup\{s_1s_2,s_1s_3,s_2s_3\}| \geq 5 \). Therefore, we have that
 \begin{align*}
  |c(K_n)|\leq & \frac{1}{2}(3k+2t)(3k+2t-1)-(6t+6(k-1)+3(k-2)+5)\\
  =&\frac{1}{2}(3k+2t-3)(3k+2t-4)+1, \text{a contradiction}.
\end{align*}
{\bf Scenario 2.16:} Let \( c(K_n[S]) \subseteq c(H) \). Moreover, suppose that \( c(y^b_1y^b_2) \in c(K_n[S]) \), where \( b \in [t] \). Under these conditions, if any of the following statements are true, we can either construct a rainbow subgraph of \( K_n \) that is isomorphic to \( kP_3 \cup tP_2 \), or we can apply one of the scenarios from 2.1 to 2.14 to achieve the desired result.
\begin{enumerate}
  \item There are $i\in[t]$, $j_1\in[3]$, and $j_2\in([3]-\{j_1\})$ such that $\{y^i_1s_{j_1},y^i_2s_{j_2}\}\subseteq E(G)$.
  \item There are $i\in([t]-\{b\})$, $i_1\in[2]$, and $i_2\in([2]-\{i_1\})$ such that $\{y^i_{i_1}y^b_1,y^i_{i_2}y^b_2\}\subseteq E(G)$.
  \item There are $i\in([t]-\{b\})$ and $i_1,i_2\in[2]$ such that $\{y^b_{i_1}y^i_1,y^b_{i_1}y^i_2\}\subseteq E(G)$ and $|\{y^i_{i_2}s_i:i\in[3]\}\cap E(G)|=3$.
  \item There are $i\in([t]-\{b\})$, $j\in[k-1]$, $i_1\in[2]$, $i_2\in([2]-\{i_1\})$, and $j_1\in\{1,3\}$ such that $\{y^b_{1}y^i_{i_1},y^b_{2}y^i_{i_1}\}\subseteq E(G)$ and $y^i_{i_2}x^j_{j_1}\in E(G)$.
   \item For each \( i \in [k-1] \), there are two distinct elements \( i_1, i_2 \in [3] \), as well as two other distinct elements \( j_1, j_2 \in [3] \), not necessarily different from $i_1$ and $i_2$,   such that \( \{ x^i_{i_1} s_{j_1}, x^i_{i_2} s_{j_2} \} \subseteq E(G) \).
  \item There are $i\in[k-1]$, $i_1\in[3]$, and $i_2\in([3]-\{i_1\})$ such that $\{y^b_1x^i_{i_1},y^b_2x^i_{i_2}\}$ $\subseteq$ $E(G)$.
  \item There is $i\in[2]$ such that $|\{y^b_{i}s_j:j\in[3]\}\cap E(G)|=3$.
  \item There are  $i\in[k-1]$, $j\in[3]$, and $j_1\in[2]$ such that $|\{y^b_{j_1}s_h:h\in[3]\}\cap E(G)|=2$ and $|\{s_jx^i_{i_1}:i_1\in[3]\}\cap E(G)|=3$.
\end{enumerate}
Assuming that none of the above conditions are true, we can observe that \( |\mathcal{S}_{2.16}(1,2,3,4)| \geq 3t + 3(t-1) \), \( |\mathcal{S}_{2.16}(5,6)| \geq 9(k-1) \), and \( |\mathcal{S}_{2.16}(7,8) \cup \{s_1s_2, s_1s_3, s_2s_3\}| \geq 5 \). Thus, we conclude that
 \begin{align*}
  |c(K_n)|\leq & \frac{1}{2}(3k+2t)(3k+2t-1)-(3t+3(t-1)+9(k-1)+5)\\
  =&\frac{1}{2}(3k+2t-3)(3k+2t-4)+1, \text{a contradiction}.
\end{align*}

Based on Scenarios 2.1 to 2.16, we can conclude that if \( c(K_n) = \frac{1}{2}(3k + 2t - 3)(3k + 2t - 4) + 2 \), then the graph \( K_n \) contains a rainbow subgraph that is isomorphic to \( kP_3 \cup tP_2 \). Therefore, we have \( AR(n, kP_3 \cup tP_2) \leq \frac{1}{2}(3k + 2t - 3)(3k + 2t - 4) + 1 \). This conclusion, along with Case 1, completes the proof of Theorem \ref{th3}.

\section{Concluding remarks}

We have determined the anti‑Ramsey number of the linear forest \(kP_{3}\cup tP_{2}\) in the spanning regime $n = 3k+2t$, $k\ge 1$, and $t\ge 2$, namely
\[
\operatorname{AR}\bigl(3k+2t,\; kP_{3}\cup tP_{2}\bigr) = \frac12\bigl(3k+2t-3\bigr)\bigl(3k+2t-4\bigr)+1.
\]
This removes the quadratic lower bound on \(t\) required in previous work \cite{new-1}, where \(t\ge \frac{k^{2}-k+4}{2}\). Hence, our theorem provides a complete characterization for the critical case in which the host graph has exactly the same number of vertices as the forest.

A natural direction for future research is to improve the result of \cite{p1-10} for the non‑spanning case \(n\ge 3k+2t+1\) when \(t\) is small. Determining \(\operatorname{AR}(n,kP_{3}\cup tP_{2})\) under these conditions remains open.

%\section*{Acknowledgments}

%This research was supported by the NSFC under the grant 12131013.

\section*{Conflicts of interest} 

The authors declare no conflict of interest.

\section*{Data availability} 

No data was used in this investigation.

\end{document}